\documentclass[12pt]{amsproc}
\newtheorem{theorem}{\sc Theorem}[section]
\newtheorem{lemma}[theorem]{\sc Lemma}

\begin{document}

\author{R. Grigorchuk}
\address{Department of Mathematics\\ Texas A\&M University\\ College Station, TX, USA}
\email{grigorch@math.tamu.edu}

\author{P. Shumyatsky}
\address{Department of Mathematics\\ University of Bras\'\i lia\\ 70910 Bras\'\i lia DF
\\ Brazil}
\email{pavel@unb.br}

\keywords{Just infinite groups, locally finite groups, $C^*$-algebras}
\subjclass{46L05, 20F50, 16S34}

\thanks{The first author was supported by NSA grant H98230-15-1-0328.
The second author was supported by FAPDF and CNPq}

\title[Just infinite periodic locally soluble groups]{On just infinite periodic locally soluble groups}

\begin{abstract} We construct an uncountable family of periodic locally soluble groups which are hereditarily just infinite. We also show that the associated full $C^*$-algebra $C^*(G)$ is just infinite for many groups $G$ in this family.
\end{abstract}
\maketitle

\section{Introduction}

An infinite group $G$ is just infinite if every proper quotient of $G$ is finite. Just infinite groups  are  on the  border  between  finite  and  infinite  groups and therefore deserve special  attention. In \cite{gri1} the first author introduced three classes of just infinite groups: branch just infinite groups, and just infinite groups that contain a subgroup of finite index which is a direct product of finitely many copies of a group $L$, where $L$ is either simple or hereditarily just infinite. According to the definition adopted in \cite{gri1} a group $L$ is hereditarily just infinite if it is infinite, residually  finite and every  subgroup of finite  index  in $L$ is just infinite. He showed, using results of Wilson \cite{wi2}, that the class of just infinite groups splits into  these three subclasses. At present time, many experts use the term ``hereditarily just infinite group" to mean ``infinite group all of whose subgroups of finite index are just infinite", that is, the group is not required to be residually finite. To avoid ambiguity, in this paper on each occasion we state explicitly that the groups we consider are hereditarily just infinite and residually finite.

In \cite{gri3} just infinite $C^*$-algebras were introduced and studied.  In particular, a trichotomy (splitting in three subclasses) was  proved,  but this trichotomy uses quite different ingredients than in \cite{gri1} (still  the simplicity and residual finiteness are involved). In \cite{gri3} a problem about existence of groups with just infinite group $C^{*}$-algebras was discussed. As a possible approach for handling the problem it was suggested to deal first with the question about existence of hereditarily just infinite, residually finite, locally finite groups (see \cite[Problem 6.10]{gri3}).

Such a group was constructed in the work of Belyaev and the authors \cite{we} thus confirming that locally finite groups with just infinite group $C^{*}$-algebras exist. On the other hand, results obtained in \cite{we} suggest that the structure of a locally finite just infinite group $G$ must be rather special. Lemma 3.1 in \cite{we} implies that such a group $G$ is never locally nilpotent, while Theorem 3.3 in \cite{we} tells us that $G$ must have infinite exponent. It is an open problem whether $G$ can satisfy a nontrivial law.

Direct limits of finite groups from Wilson's Construction B in \cite{wi22} provide examples of hereditarily just infinite groups which are periodic, residually finite and locally soluble. In the present article we develop another construction of such groups, leading to Theorem \ref{baba}. It seems that our construction is of independent interest as, for example, the groups from  Wilson's Construction B have only elements of odd order. The groups emerging from our construction can have elements of any prime order. We discover an uncountable family of hereditarily just infinite, residually finite, periodic locally soluble groups and show that $C^*(G)$ is just infinite for many groups $G$ in this family (see Theorems \ref{kg} and \ref{gla} in the next section). Recall that the full $C^*$-algebra $C^*(G)$ associated with the group $G$ is a  $C^*$-algebra obtained by completion of the complex  group algebra  $\Bbb C[G]$ with respect to the norm which is the supremum of  norms given  by all  unitary representations of $G$. In the case when $G$ is amenable, i.e. has a left invariant mean (groups considered in this  paper are amenable because they  are locally  finite) the $C^*$-algebra  $C^*(G)$ coincides  with the reduced $C^*$-algebra $C_r^*(G)$ which is obtained by completion  of $\Bbb C[G]$ by  the   norm  given by the left regular representation of  $G$ in the Hilbert space $l^2(G)$. The book \cite{mur} provides an introduction to $C^*$-algebras.

We thank the referee for corrections and valuable suggestions concerning the preliminary version of the article.

\section{Results} If $X$ and $Y$ are nonempty subsets of a group $G$, we denote by $[X,Y]$ the subgroup generated by all commutators $[x,y]$, where $x\in X$ and $y\in Y$. Suppose $\alpha$ is an automorphism of $G$. We write $x^\alpha$ for the image of $x$ under $\alpha$.  If $A$ is a group of automorphisms of $G$ and $H\leq G$, the subgroup generated by elements of the form $h^{-1}h^\alpha$ with $h\in H$ and $\alpha\in A$ is denoted by $[H,A]$. It is well-known that the subgroup $[G,A]$ is an $A$-invariant normal subgroup in $G$. We also write $C_H(A)$ for the fixed point subgroup of $A$ in $H$. In the case where $A=\langle\alpha\rangle$ is cyclic, generated by $\alpha$ we write $[H,\alpha]$ and $C_G(\alpha)$, respectively. When convenient, we identify elements of the group $G$ with the corresponding inner automorphisms of $G$. As usual, the symbols $G'$ and $Z(G)$ stand for the derived group and the center of a group $G$, respectively.

We will require the following well-known fact.
\begin{lemma}\label{zuzu} Let $G$ be a metabelian group and $a\in G$. Suppose that $B$ is an abelian subgroup in $G$ containing $G'$. Then $[B,a]$ is normal in $G$.
\end{lemma}
\begin{proof} Since $B$ contains $G'$, it follows that $B$ is normal in $G$. Let $x\in G$. Write $a^x=ay$ with $y\in G'$. Taking into account that $[B,y]=1$, we obtain $[B,a]^x=[B,a^x]=[B,ay]=[B,a]$. Thus, an arbitrary element $x\in G$ normalizes $[B,a]$, as required.
\end{proof}

A coprime automorphism of a finite group $G$ is an automorphism of order relatively prime to the order of $G$. Throughout, we use the following standard facts on coprime groups of automorphisms without explicit references (\cite[Theorem 6.2.2, Theorem 5.3.5, Theorem 5.3.6]{go}).
\begin{lemma}\label{11} Let $A$ be a group of automorphisms of the finite group $G$ with $(|A|,|G|)=1$.   
\begin{enumerate}
\item If $N$ is any $A$-invariant normal subgroup of $G$, we have $C_{G/N}(A)=C_G(A)N/N$;
\item $G=[G,A]C_G(A)$;
\item If $G$ is abelian, then $G=[G,A]\times C_G(A)$;
\item $[[G,A],A]=[G,A]$.
\end{enumerate}
\end{lemma}

\begin{lemma}\label{gran} Let $G$ be a finite group admitting a coprime  automorphism $\alpha$ such that $G'\leq C_G(\alpha)$. Set $Z=Z([G,\alpha])$. Then $[Z,\alpha]\leq Z(G)$.
\end{lemma}
\begin{proof} Let $g\in G$ and $h\in C_G(\alpha)$. Taking into account that $[h,g]\in C_G(\alpha)$, write $$h^g=(h^g)^\alpha=h^{g^\alpha}.$$ We therefore deduce that $[h,[g,\alpha]]=1$. This shows that $C_G(\alpha)$ commutes with $[G,\alpha]$. Now the lemma follows from the fact that $G=[G,\alpha]C_G(\alpha)$.
\end{proof}

\begin{lemma}\label{elem} Let $G=G_1\oplus G_2\oplus\dots\oplus G_m$ be a finite abelian group written as a direct sum of isomorphic subgroups $G_i$. Suppose that $G$ admits a coprime  automorphism $\alpha$, of order $m$, which cyclically permutes the summands $G_i$. Then the index of $[G,\alpha]$ in $G$ is precisely the order of $G_1$.
\end{lemma}
\begin{proof} This is immediate from the facts that $G=[G,\alpha]\oplus C_G(\alpha)$ and $C_G(\alpha)=\{x+x^\alpha+\dots+x^{\alpha^{m-1}}\mid x\in G_1\}$ is just the diagonal subgroup of $G$.
\end{proof}

Let $p$ be a prime. Recall that a finite $p$-group $G$ is extra-special if $G'=Z(G)$ is of order $p$. In the present paper we will work with central products of such groups. We use the term ``central product" as in \cite{lgmk}: a group $G$ is a central product of its subgroups $H$ and $K$, written $G=H\circ K$, if $G$ is generated by $H$ and $K$ such that $[H,K]=1$ and $H\cap K=Z(H)=Z(K)$.

\begin{lemma}\label{extraelem} Let $m\geq2$, and let $G=G_1\circ G_2\circ\dots\circ G_m$ be a finite group which is a central product of isomorphic extra-special subgroups $G_i$ of order $p^3$. Suppose that $G$ admits a coprime  automorphism $\alpha$, of order $m$, which cyclically permutes the factors $G_i$ such that $G'\leq C_G(\alpha)$. Then $[G,\alpha]$ is extra-special and has index $p^2$ in $G$.
\end{lemma}
\begin{proof} It is clear that $G$ is extra-special. If $[G,\alpha]$ is abelian, then by Lemma \ref{gran} we have $[G,\alpha]\leq Z(G)$. This leads to a contradiction since $Z(G)=G'\leq C_G(\alpha)$. Hence $[G,\alpha]$ is nonabelian and $[G,\alpha]'=G'$. Note that $\alpha$ has no nontrivial fixed-points in the quotient $[G,\alpha]/G'$. Therefore $C_{[G,\alpha]}(\alpha)\leq G'$ and we conclude that $C_{[G,\alpha]}(\alpha)$ has order $p$. Further, Lemma \ref{gran} shows that $Z([G,\alpha])\leq C_G(\alpha)$ and so $Z([G,\alpha])$ has order $p$ as well. Therefore $[G,\alpha]$ is extra-special.

Let the automorphism of $G/G'$ induced by $\alpha$ be denoted again by the symbol $\alpha$. Lemma \ref{elem} shows that the index of $[G/G',\alpha]$ in $G/G'$ is precisely $p^2$. Therefore, back in $G$, the index of $[G,\alpha]G'$ is $p^2$ and since $G'\leq [G,\alpha]$, we conclude that the index of $[G,\alpha]$ is $p^2$.
\end{proof}

Given two groups $K$ and $M$, we write $M\wr K$ for the wreath product of $M$ by $K$. Thus, $M\wr K$ is a semidirect product of a normal subgroup $B$, isomorphic to the direct product of $|K|$ copies of $M$, and the subgroup isomorphic to $K$ acting on $B$ by permuting the direct factors in the natural way. Thus, we will identify $K$ with its isomorphic image in $M\wr K$. 

Suppose that $M$ is an extra-special group of order $p^3$ for some prime $p$. It is clear that $Z(B)=B'$. Obviously, the subgroup $[B',K]$ is normal in $M\wr K$. Set $B_0=B/[B',K]$ and let $H$ be the natural semidirect product of $B_0$ by $K$. Of course, $H$ is isomorphic to the quotient group $(M\wr K)/[B',K]$. Recall that $B=M_1\times M_2\times\dots\times M_{|K|}$, where the factors $M_i$ are isomorphic to $M$. Since $K$ permutes the commutator subgroups $M_i'$ transitively and since the subgroup $[B',K]$ is $K$-invariant, the fact that $B_0$ is nonabelian implies that no subgroup $M_i'$ is contained in $[B',K]$. Thus, the images of $M_i$ in $B_0$ are isomorphic to $M_i$. We will therefore regard the groups $M_i$ from now on as subgroups of $B_0$. Let $x\in Z(B_0)$. Write $x=x_1x_2\dots x_{|K|}$ for suitable $x_i\in M_i$. Suppose that $x_1\not\in B_0'$. Then we can choose $y\in M_1$ such that $[x_1,y]\neq1$. Then we also have $[x,y]\neq1$, which contradicts the assumption that $x\in Z(B_0)$. Therefore $x_1\in B_0'$. Similarly, we deduce that $x_i\in B_0'$ for all $i$ and therefore $Z(B_0)=B_0'$. It follows that $B_0$ is a central product of the groups $M_i$. Let $a\in K$ be a $p'$-element and set $D=[B_0,a]$. In view of Lemma \ref{extraelem} we deduce that $D$ is extra-special and $D'=Z(D)=B_0'$. After these preparations we now describe our main construction.

Let $p_1,p_2,\dots$ be an infinite sequence of primes such that $p_{i+1}\neq p_i$ for all $i=1,2,\dots$ (but we allow $p_{i+2}=p_i$). We will construct an infinite sequence of finite soluble groups $G_1,G_2,\dots$ such that for each $i=1,2,\dots$ the following holds.
\begin{itemize} \item[(a)] There is a natural embedding $G_i\leq G_{i+1}$.
\item[(b)] $G_i$ has a unique minimal normal subgroup $R_i$; 
\item[(c)] $R_i\leq Z(G_i)$;
\item[(d)] $R_i$ is of order $p_i$.
\end{itemize}
Let $G_1=R_1$ be the cyclic group of order $p_1$. Now assume that $i\geq2$ and the group $G_{i-1}$ satisfies the above conditions. Let $M_i$ be an extra-special $p_i$-group of order $p_i^3$ and $H=M_i\wr G_{i-1}$. Denote by $B$ the base group of this wreath product. We view $G_{i-1}$ as a subgroup of $H$. Of course, $[B',H]=[B',G_{i-1}]$. The index of $[B',G_{i-1}]$ in $B'$ is exactly $p_i$. The group $G_{i-1}$ naturally acts on the quotient group $B_i=B/[B',H]$. Since $R_{i-1}\leq Z(G_{i-1})$, it follows that $D_i=[B_i,R_{i-1}]$ is normalized by $G_{i-1}$. Let $G_i$ be the semidirect product of $D_i$ by $G_{i-1}$, with the induced action. Since $R_{i-1}$ is the unique minimal normal subgroup in $G_{i-1}$ and since $R_{i-1}$ does not centralize $D_i$, it follows that $G_{i-1}$ faithfully acts on $D_i$ and therefore each normal subgroup of $G_i$ has nontrivial intersection with $D_i$. Taking into account that $D_i'$ is a unique minimal normal subgroup in $D_i$ we conclude that $D_i'$ is a unique minimal normal subgroup in $G_i$. Therefore we can define $R_i=D_i'$. Thus, the group $G_i$ has the required properties.

For indices $2\leq k\leq i$ we define the subgroup $T_{ik}$ in $G_i$ by setting $T_{ik}=D_iD_{i-1}\dots D_k$. It is clear that $T_{ii}=D_i$.

\begin{lemma}\label{lala} For a fixed $i$ we have.
\begin{itemize} 
 \item[(a)] The subgroups $T_{ik}$ are normal in $G_i$. 
\item[(b)] Each subgroup $T_{ik}$ is complemented in $G_i$ by the subgroup $G_{k-1}$.
\item[(c)] $[T_{ik},R_{k-1}]=T_{ik}$.
\end{itemize}
\end{lemma}
\begin{proof} The lemma is quite obvious when $i=2$. In the general case all claims can be obtained by an easy induction on $i$. 
\end{proof}

Let $G$ be the union (direct limit) of the above sequence $G_1\leq G_2\leq \dots$. 
\begin{theorem}\label{baba} The group $G$ is a hereditarily just infinite, residually finite,  periodic, and locally soluble group.
\end{theorem}
\begin{proof} That $G$ is periodic and locally soluble is obvious. For each $k\geq2$ we denote by $T_k$ the union of all $T_{ik}$ with $i=k,k+1,\dots$. Since $T_{ik}$ is normal in $G_i$, it follows that $T_k$ is normal in $G$. Moreover, $T_k$ is complemented in $G$ by $G_{k-1}$. Therefore each $T_k$ has finite index in $G$. Suppose that $x$ is a nontrivial element lying in $T_k$ for each $k$. Then $x$ does not belong to $G_{k-1}$ for each $k$. This is a contradiction since $G$ is the union of the sequence $G_1\leq G_2\leq \dots$. Hence, the intersection of all $T_k$ is trivial and so $G$ is residually finite.

Let $N$ be a proper nontrivial normal subgroup of $G$. Let $j$ be the minimal index such that $R_j\leq N$. For any $i>j$ the normal closure of $R_j$ in $G_i$ is contained in $N$. In view of Lemma \ref{lala}(c) we conclude that $T_{j+1}$ is contained in $N$, as well. Since the index of $T_{j+1}$ in $G$ is finite, so is the index of $N$. Thus, $G$ is just infinite and every subgroup of finite index contains the subgroup $T_k$ for some $k$.
Hence, once we show that every normal subgroup of $T_k$ has finite index, we will be able to conclude that $G$ is a hereditarily just infinite group.

Set $T=T_k$ and let $S$ be a nontrivial normal subgroup of $T$. Recall that $T$ is the union of the finite subgroups $T_{ik}$ with $i=k,k+1,\dots$. Let $j$ be an index such that $S\cap T_{jk}\neq1$. We have $S\cap D_j\neq1$. Indeed, set $S_0=S\cap T_{jk}$. Note that both subgroups $S_0$ and $D_j$ are normal in $T_{jk}$ and $D_j$ contains its centralizer in $T_{jk}$. Thus, either $S_0\leq D_j$ or $1\neq[S_0,D_j]\leq S\cap D_j$ whence $S\cap D_j\neq1$. Taking into account that $R_j=D'_j$ is a unique normal subgroup of $D_j$ we conclude that $R_j$ is contained in $S$. Invoking Lemma \ref{lala}(c) deduce that $T_{i(j+1)}$ is contained in $S$ for any $i\geq j+1$. Therefore the subgroup of finite index $T_{(j+1)}$ is contained in $S$. Hence, $S$ has finite index in $G$.
\end{proof}

We will now have a closer look at the action of $G_{i-1}$ on $D_i$. Recall that $D_i=[B_i,R_{i-1}]$, where $B_i=A_1\circ A_2\circ\dots\circ A_{|G_{i-1}|}$ is a central product of extra-special groups $A_j$ of order $p_i^3$, each isomorphic to $M_i$, and the group $G_{i-1}$ regularly permutes the factors. Let $O_1,O_2,\dots,O_m$ be the orbits under the permutational action of $R_{i-1}$ on the set $\{A_1,A_2,\dots,A_{|G_{i-1}|}\}$. For each $l\leq m$ set $$E_l=[\prod_{A_j\in O_l}A_j,R_{i-1}].$$ By Lemma \ref{extraelem}, each $E_j$ is extra-special. 
Thus, $$D_i=E_1\circ E_2\circ\dots\circ E_l \eqno{(*)}$$ is a central product of isomorphic extra-special groups $E_j$. We see that $R_{i-1}$ normalizes each of the factors $E_l$ and the natural permutational action of $G_{i-1}/R_{i-1}$ on the set $\{E_1,E_2,\dots,E_l\}$ is a regular one. In what follows this fact will play an important role. We will also need the following observation.

\begin{lemma}\label{dada} Each of the subgroups $E_j$ contains non-central elements of prime order.
\end{lemma}
\begin{proof} Set $E=E_j$ and $R=R_{i-1}$. Thus, by Lemma \ref{extraelem}, $E$ is an extra-special $p_i$-group of index $p_i^2$ in $\prod_{A_j\in O_l}A_j$. Further, $E$ is acted on by the group $R$, of order $p_{i-1}$, in such a way that $E=[E,R]$. Suppose that $E$ has no non-central elements of prime order. Then of course, $E$ is isomorphic to the quaternion group of order 8 and $R$ has order 3. On the other hand, since $R$ has order 3, the group $\prod_{A_j\in O_l}A_j$ is a central product of 3 extraspecial groups of order $8$. Thus $\prod_{A_j\in O_l}A_j$ has order $2^7$. Being a subgroup of index 4, the subgroup $E$ must have order $2^5$, a contradiction.
\end{proof}

A remarkable property of our groups $G$ is that in many cases all nonzero ideals in the group algebra $K[G]$ have finite codimension. More precisely, we have the following theorem.

\begin{theorem}\label{kg} Let $K$ be a field of characteristic $p\geq0$. Suppose that the sequence of primes $p_1,p_2,\dots$ has the properties that $p_{i+1}\neq p_i$ for all $i$ and for any index $i$ there exist indices $u,v$ such that $i<u<v$ and $p_u=p_v\neq p$. Then every nonzero ideal in the group algebra $K[G]$ has finite codimension.
\end{theorem}

Before we embark on the proof, note that many sequences of primes satisfy the conditions in the theorem. For instance, this happens whenever a prime $q$, different from $p$, occurs in the sequence infinitely many times, or whenever infinitely many primes occur in the sequence more than once.
\begin{proof}[Proof of Theorem \ref{kg}] Let $I$ be a nonzero ideal in $K[G]$ and $0\neq\alpha\in I$. Then we can assume that 1 occurs in the support of $\alpha$ and we write $\alpha=\sum_0^n k_ix_i^{-1}$, where $1=x_0,x_1,\dots,x_n$ are distinct elements of $G$ and  the coefficients $k_i$ lie in $K$. Note that $k_0\neq0$. Let $t$ be the minimal index such that $x_0,x_1,\dots,x_n\in G_t$. Set $X=\langle x_0,x_1,\dots,x_n\rangle$ and remark that $X\cap R_{t+j}=1$ for any $j\geq1$. Choose indices $v>u>t+1$ such that $p_u=p_v=q\neq p$ and consider the decomposition $D_u=E_1\circ E_2\circ\dots\circ E_l$ as in (*). Since $X\cap R_{u-1}=1$, it follows that $X$ regularly permutes the factors $E_i$. By Lemma \ref{dada} $E_1$ contains a non-central element $y$ of order $q$. For $i=0,1,\dots,n$ set $y_i=y^{x_i}$ and $z_i=[y_i,x_i]$. Denote the subgroup $\langle y_i^{x_j};0\leq i,j\leq n\rangle$ by $A$. Since no two of the elements $y_i^{x_j}$ belong to the same factor $E_k$, it follows that $A$ is abelian and the subgroup $\langle z_1,\dots,z_n\rangle$ is elementary abelian of order $q^n$.

We show now by inverse induction on $s$ with $n\geq s\geq0$ that $I$ contains an element $$\beta_s=\sum_0^s\beta_{si}x_i^{-1}$$ such that $\beta_{si}\in K[A]$ and, if $s<n$, we have $$\beta_{s0}=k_0(z_n-1)(z_{n-1}-1)\dots(z_{s+1}-1).$$ First for $s=n$ we merely take $\beta_s=\alpha$. Now suppose we have $\beta_s$ as above contained in $I$ with $n>s>0$. Then $y_s^{-1}\beta_sy_s$ and $z_s\beta_s$ both belong to $I$ and hence
$$\beta_{s-1}=z_s\beta_s-y_s^{-1}\beta_sy_s\in I.$$
Since $A$ is abelian and $y_s,z_s\in A$, we have
$$\beta_{s-1}=\sum_0^s z_s\beta_{si}x_i^{-1}-\sum_0^sy_s^{-1}\beta_{si}x_i^{-1}y_s$$ $$=\sum_0^s\beta_{si}(z_s-[y_s,x_i])x_i^{-1}=\sum_0^{s-1}\beta_{si}(z_s-[y_s,x_i])x_i^{-1}$$ since $z_s=[y_s,x_s]$. Recall that $x_0=1$. It follows that $\beta_{s-1,0}$ has the required form and the induction step is proved.

In particular, when $s=0$ we conclude that $$\beta_0=k_0(z_n-1)(z_{n-1}-1)\dots(z_1-1)\in I.$$ Here $k_0\neq0$ and, since $\langle z_1,z_2,\dots,z_n\rangle=\langle z_1\rangle\times\langle z_2\rangle\times\dots\times\langle z_n\rangle$, we conclude that $\beta_0\neq0$. Thus, we have shown that $I\cap K[A]\neq0$.

Write $a_0=1$ and $A=\{a_0,a_1,\dots,a_m\}$. We know that $A\cap R_u=1$ and so $A\cap R_{u+j}=1$ for any $j\geq0$. Consider now the decomposition $D_v=\tilde E_1\circ \tilde E_2\circ\dots\circ \tilde E_{l'}$ as in (*). The group $A$ regularly permutes the factors $\tilde E_j$. In view of Lemma \ref{dada} we can choose a non-central element $b$ of order $q$ that belongs to $\tilde E_1$. For $i=0,1,\dots,m$ set $b_i=b^{a_i}$. Denote the subgroup $\langle b_0,b_1,\dots,b_m\rangle$ by $B$. Set $H=BA$. The choice of the element $b$ guarantees that $H\cong C_q\wr A$, where $C_q$ denotes the cyclic group of order $q$. In particular, the center $Z(H)$ is cyclic generated by the product $b_0\cdots b_m$. 

Fix $i\leq m$ and set $B_i=[B,a_i]$. Since $B$ is abelian, $B_i$ is isomorphic to $B/C_B(a_i)$. Let $\hat B_i$ denote the sum of elements in $B_i$. We compute $$\sum_{b\in B}(a_i^{-1})^b=\sum_{b\in B}[b,a_i]a_i^{-1}=|C_B(a_i)|\hat B_ia_i^{-1}.$$

Let $\alpha_0=\sum l_ia_i^{-1}\in I\cap K[A]$. Here $l_i$ are elements of $K$ and note that $l_0\neq0$. Then $\beta=\sum_{b\in B}b^{-1}\alpha_0b\in I$. By the above $$\beta=\sum_il_i|C_B(a_i)|\hat B_ia_i^{-1}.$$ For $i=0$, since $a_0=1$, the summand here is $l_0|B|$. For $i\neq0$ the element $a_i$ does not centralize $B$ so $B_i\neq1$. Lemma \ref{zuzu} tells us that each subgroup $B_i$ is normal in $H$. Recall that $Z(H)$ is cyclic of order $q$. Because $H$ is nilpotent, every normal subgroup of $H$ has nontrivial intersection with $Z(H)$. Consequently, $Z(H)\leq B_i$. Let $z$ be a generator of $Z(H)$. Since $z\in B_i$, we have $(z-1)\hat B_i=0$ for $i\neq0$. Therefore $$(z-1)\beta=l_0|B|(z-1).$$ We put this together with the facts that $(z-1)\beta\in I$ and $l_0\neq0$. Since the characteristic of $K$ is not $q$, it follows that $z-1\in I$.

Now set $N=\{g\in G\mid g-1\in I\}$. This is a normal subgroup in $G$. The above shows that $N\neq1$. Since $G$ is just infinite, we conclude that $N$ has finite index. This implies that $I$ has finite codimension.
\end{proof}

Combining Theorem \ref{kg} with results from \cite{gri3} we can now easily show that whenever $G$ is as in Theorem \ref{kg} with $p=0$ the $C^*$-algebra $C^*(G)$ is just infinite. We denote by $\Bbb C$ the field of complex numbers.
\begin{theorem}\label{gla} Assume that the sequence of primes $p_1,p_2,\dots$ has the properties that $p_{i+1}\neq p_i$ for all $i$ and for any index $i$ there exist indices $u,v$ such that $i<u<v$ and $p_u=p_v$. Then $C^*(G)$ is just infinite.
\end{theorem}
\begin{proof} By Theorem \ref{kg}, the group algebra $\Bbb C[G]$ is just infinite. The result now follows immediately from Propositions 6.2 and 6.6 in \cite{gri3}.
\end{proof}


\begin{thebibliography}{99}
\bibitem{we} V. Belyaev, R. Grigorchuk, P. Shumyatsky, On just infiniteness of locally finite groups and their $C^*$-algebras, Bull. Math. Sci., to appear, doi:10.1007/s13373-016-0091-4.
\bibitem{go} D. Gorenstein, Finite Groups, Harper and Row, New York, 1968.
\bibitem{gri1} R. I. Grigorchuk, Just infinite branch groups. New horizons in pro-$p$ groups, 121--179, Progr. Math., {\bf 184}, Birkh\"auser Boston, Boston, MA, 2000.
\bibitem{gri3} R. Grigorchuk, M. Musat, M. R\o rdam, just infinite $C^*$-algebras, Preprint, 2016, https://arxiv.org/abs/1604.08774.
\bibitem{lgmk} C. R. Leedham-Green,  S. McKay, The structure of groups of prime power order, London Mathematical Society Monographs. New Series, 27 (2002), Oxford University Press.
\bibitem{mur} G. Murphy, C*-algebras and operator  theory, Academic Press, Boston, MA, 1990.
\bibitem{wi2} J. S. Wilson, Groups with every proper quotient finite, Proc. Cambridge Philos. Soc., {\bf 69} (1971), 373--391.
\bibitem{wi22} J. S. Wilson, Large hereditarily just infinite groups, J. Algebra, {\bf 324} (2010), 248--255.
\end{thebibliography}
\end{document}